\documentclass[12pt,a4paper,reqno]{amsart}
\usepackage[utf8x]{inputenc}
\usepackage[mathcal]{eucal}
\usepackage{amsmath,amssymb,amsthm}
\usepackage[nice]{nicefrac}
\usepackage[dvipsnames]{xcolor}
\usepackage[colorlinks=true, linkcolor=MidnightBlue, citecolor=OliveGreen]{hyperref}
\usepackage[normalem]{ulem}
\usepackage{mathrsfs}
\usepackage{mathtools}

\theoremstyle{plain}
\newtheorem{theorem}{Theorem}[section]
\newtheorem{lemma}[theorem]{Lemma}
\newtheorem{proposition}[theorem]{Proposition}

\theoremstyle{definition}
\newtheorem{definition}[theorem]{Definition}

\newtheorem{remark}[theorem]{Remark}

\numberwithin{equation}{section}

\DeclareMathOperator{\Aut}{Aut}
\DeclareMathOperator{\St}{St}

\DeclareMathOperator{\Rst}{Rst}

\DeclareMathOperator{\Sym}{Sym}
\DeclareMathOperator{\Norm}{N}

\DeclareMathOperator{\F}{\mathbb{F}}
\DeclareMathOperator{\rank}{rk}
\DeclareMathOperator{\N}{\mathbb N}
\DeclareMathOperator{\Z}{\mathbb Z}

\newcommand{\cref}[3][]{\hyperref[#3]{#2~\ref*{#3}#1}}

\title[Conjugacy classes of polyspinal groups]{Conjugacy classes of polyspinal groups}

\author[J.\,M. Petschick]{Jan Moritz Petschick} 
\address{Jan Moritz Petschick: Mathematisches Institut, Heinrich-Heine-Universit\"at, 40225
  D\"usseldorf, Germany} \email{Jan.Petschick@hhu.de}

\author[A. Thillaisundaram]{Anitha Thillaisundaram}  
\address{Anitha Thillaisundaram: Centre for Mathematical Sciences, Lund University,  223 62 Lund, Sweden}
\email{anitha.t@cantab.net}

\date{\today}
 
\begin{document}
\begin{abstract}
Spinal groups and multi-GGS groups are both generalisations of the well-known Grigorchuk-Gupta-Sidki (GGS-)groups. Here we
give a necessary condition for spinal groups to be conjugate, and we establish a necessary and sufficient condition for multi-GGS groups to be conjugate. 
We also introduce a natural common generalisation of both classes, which we call polyspinal groups.
Our results enable us to give a negative answer to a question of Bartholdi, Grigorchuk and \u{S}uni\'{k}, on whether every finitely generated branch group is isomorphic to a weakly branch spinal group.
\end{abstract}

\maketitle

\section{Introduction}

Let $m\in\N_{\ge2}$  and let $T=T_m$ be the $m$-adic tree. Groups acting on $m$-adic trees have received quite a bit of attention, especially in the case when $m$ is a prime. The interest in these groups is largely due to their nice structure, their importance in the theory of just infinite groups, and the fact that many such groups have exotic algebraic properties; we refer the reader to~\cite{BGS03} for a good introduction.

The (first) Grigorchuk group~\cite{Grigorchuk} was the first notable group acting on an $m$-adic tree that was constructed, and it continues to play a central role in the subject.
 It is a 3-generated infinite periodic group acting on the binary rooted tree~$T_2$ with many interesting properties. Its three generators include the rooted automorphism~$a$ which swaps the two maximal subtrees of $T_2$, and two directed automorphisms $\beta$ and $\gamma$, both of which stabilise the rightmost infinite ray of the tree. They are defined recursively as follows:
\[
\beta=(a,\gamma),\quad \gamma=(a,\delta),
\]
where for $x$ and $y$ automorphisms of $T_2$, the notation $(x,y)$ indicates the independent actions on the respective maximal subtrees, and
\[
\delta=(1,\beta).
\]
Soon after the Grigorchuk group was defined, other similar constructions followed, including the well-studied classes of Grigorchuk-Gupta-Sidki (GGS-)groups and 
\u{S}uni\'{k} groups (also called siblings of the Grigorchuk group).

These constructions were generalised to a natural class of so-called \emph{spinal groups}, 
which are generated by a group~$R$ of rooted automorphisms and a group~$D$ of directed automorphisms. Both are defined by restricting the area of the tree where the automorphisms act non-trivially; rooted automorphisms act only at the root of the tree, while directed automorphisms only act on a $1$-sphere around a constant path.
We refer the reader to \cref{Section}{sec:multi-spinal} for the precise definition. To any group~$D$ of directed automorphisms, there is an associated sequence of homomorphisms $\omega=(\omega_n)_{n\in\N}$ from $D$ to $\Sym(X)$, prescribing the action of $D$ on the different levels of the tree. This sequence fully determines the directed automorphisms. We write $\sigma^nR$ for the group generated by the images of the associated rooted automorphisms at level~$n$.
Note also that we can identify the vertices of the $m$-adic tree~$T$ with the elements of the free monoid~$X^*$, for the alphabet $X=\{0,1,\ldots,m-1\}$; see \cref{Section}{sec:prelim} for precise details.

Recently, Petschick~\cite{Moritz} identified the conjugacy classes of GGS-groups within the automorphism group of their respective trees. In this paper, we give a condition for spinal groups to be conjugate in the same sense.

\begin{theorem}[A necessary condition for spinal groups to be conjugate]\label{thm:spinal}
	Let $G$ and $\widetilde G$ be spinal groups with defining data $R,\, \omega$, respectively $\widetilde R,\, \widetilde\omega$, and let $f \in \Aut T$ be such that $\widetilde G^f = G$. Then there is an integer $N$ and an isomorphism $\iota: D \to \widetilde D$, such that for all $n \geq N$ there is:
	\begin{enumerate}
		\item an inner automorphism $\phi_n $ of $\St_{\Sym(X)}(0)$ such that $\phi_n (\sigma^n  \widetilde R) = \sigma^n R$,
		\item a tuple of inner automorphisms $\rho_n \in (\operatorname{Inn}(\sigma^n  \widetilde R))^{X\setminus\{0\}}$,
		\item an automorphism $\alpha_n $ of $\Sym(X)^{X\setminus\{0\}}$ permuting the direct factors by an element $\alpha' \in \St_{\Sym(X)}(0)$, 
	\end{enumerate}
	such that
	\[
		\omega_n = \phi_n^{\,X\setminus\{0\}} \circ \rho_n \circ \alpha_n \circ \widetilde \omega_n \circ \iota.
	\]
\end{theorem}

A subclass of spinal groups, resembling the GGS-groups more closely, have received increased attention in recent times. These are called \emph{multi-GGS groups}, which are generated by a rooted automorphism permuting the maximal subtrees cyclically, and the sequences defining their directed generators are constant, but – in contrast to the GGS-groups – they have possibly more than one directed generator. Up to relabelling the vertices of the tree, all multi-GGS groups on a fixed tree have the same rooted group $A_m = \langle (0\; 1\;\cdots \; m-1) \rangle \cong C_m$. These groups were first defined in~\cite{AKT}, where they were originally called multi-edge spinal groups. We prefer the term multi-GGS groups, as the original name can be easily confused with the term \emph{multispinal group}, which represents a more generalised family of groups generated by rooted and tree-like automorphisms, as defined in~\cite{SS}. We also consider the class of \emph{multi-EGS groups}, which allows for directed generators associated to different constant paths, and includes branch groups that do not have the congruence subgroup property. To deal with both the classes of multi-EGS and spinal groups simultaneously, we introduce a common natural generalisation, which we call \emph{polyspinal groups}; see \cref{Section}{sec:multi-spinal} for  details.

For multi-GGS groups, we are able to provide a necessary and sufficient condition for the groups to be conjugate.

\begin{theorem}[A condition for multi-GGS groups to be conjugate]\label{thm:mGGS}
	Let $G$ and $\widetilde G$ be multi-GGS groups with defining data $\omega$ and $\widetilde \omega$ respectively. There is an element $f \in \Aut T$ such that $\widetilde G^f = G$ if and only if there exists
	\begin{enumerate}
		\item an automorphism $\alpha$ of $A_m^{\,m-1}$ permuting the direct factors by an element of $\Norm_{\Sym(X)}(A_m)\cap\St_{\Sym(X)}(0)$ and
		\item an isomorphism $\iota:D\to \widetilde D$ such that
	\end{enumerate}
	\[
		\omega = \alpha \circ\widetilde \omega \circ  \iota.
	\]
\end{theorem}

It was asked by Bartholdi, Grigorchuk and \u{S}uni\'{k} \cite[Question 4]{BGS03} whether every finitely generated branch group is isomorphic to a spinal group. Using the above results, we give a negative answer to this question within the class of weakly branch groups. We refer the reader to \cref{Section}{sec:prelim} for the definitions of  weakly branch and branch groups.

\begin{theorem}\label{thm:non-spinal}
	There exists a finitely generated branch group $G \leq \Aut T_3$ such that $G$ is not isomorphic to any spinal group $S \leq \Aut T_3$. Also, if $G$ is isomorphic to a spinal group $S \leq \Aut \widetilde T$ on a different tree~$\widetilde T$, then $S$ is not weakly branch with respect to its embedding into~$\Aut \widetilde T$.
\end{theorem}

\medskip

\noindent \textit{Organisation.}  \cref{Section}{sec:prelim} consists of background material on groups acting on the $m$-adic tree and the definitions of weakly branch, branch, spinal, and polyspinal groups. In \cref{Section}{sec:conditions} we prove \cref{Theorems}{thm:spinal} and \ref{thm:mGGS}, and in \cref{Section}{sec:non-spinal} we prove \cref{Theorem}{thm:non-spinal}.

\subsection*{Acknowledgements} The first author acknowledges support by the Deutsche Forschungsgemeinschaft (German Research Foundation), grant 380258175. The second author acknowledges  support from EPSRC, grant EP/T005068/1, and from the Heinrich-Heine-Universit\"at (HHU) Forscher Alumni Programm.  She thanks 
 HHU for its
 hospitality.
Both authors thank B.~Steinberg for pointing out the family of multispinal groups.

\section{Preliminaries}\label{sec:prelim}
  For two integers $l, u \in \Z$, we denote by $[l, u]$ and $[l,u)$ the set
  of integers within the respective intervals.

  Let $m\in\N_{\ge 2}$ and let $T = T_m$ be the $m$-adic  tree,
  that is, a rooted tree where all vertices have $m$~children. Using the
  alphabet $X = [0, m)$, we identify $T$ with the
  Cayley graph of the free monoid~$X^*$ with respect to $X$, by identifing the root of the tree with the empty word. Thereby we establish
  a natural length function on~$T$. We will routinely refer to vertices of the tree as words, using the said identification.
  The words $u$ of length $\lvert u \rvert = n$,
(i.e. vertices of distance $n$ from the root)
  are called the $n$\textsuperscript{th} level vertices and 
  constitute the \textit{$n$\textsuperscript{th} layer} of the tree.

  By $T_u$ we denote the full rooted subtree of $T$ that has its root at
  a vertex~$u$ and includes all vertices having~$u$ as a prefix.
  For any two vertices $u$ and $v$, the map induced by replacing the
  prefix $u$ by $v$, yields an isomorphism between the
  subtrees $T_u$ and~$T_v$.  

  Every $f\in \Aut T$ fixes the root, and the orbits of
  $\Aut T$ on the vertices of the tree~$T$ are precisely its
  layers.
  For $u \in X^*$ and
  $x \in X$ we have $f(u x) = f(u) x'$, for $x' \in X$ 
  uniquely determined by~$u$ and~$f$. This induces a permutation
  $f|^{u}$ of~$X$ which satisfies
  \[
  f(u x) = f(u) f|^{u}(x).
  \]
 The permutation~$f|^{u} \in \Sym(X)$ is called the \textit{label} of~$f$ at~$u$, and the collection of all labels of~$f$ constitutes the \emph{portrait} of~$f$. There is a one-to-one correspondence between automorphisms of~$T$ and portraits.
We say that an automorphism $f\in \text{Aut }T$  has \emph{constant portrait} induced by a permutation~$\sigma$ of~$X$ if  all labels of~$f$ equal~$\sigma$; this automorphism is denoted by~$\kappa(\sigma)$.
 
 The automorphism $f$ is \textit{rooted} if $f|^\omega = 1$ for
  $\omega$ unequal to the root. Rooted automorphisms can be thought of as both elements of $\Aut T$ and $\Sym(X)$.
  
  Let $x \in X$ be a letter. We write $\overline{x}$ for the infinite simple rooted path $(x^n)_{n \in \N_0}$.
  The automorphism $f$ is \textit{directed}, with directed
  path~$\overline{x}$ for some $x \in X$, if the support
  $\{ \omega \mid f|^\omega \ne 1 \}$ of its labelling is infinite and
  marks only vertices at distance~$1$ from the set of
  vertices corresponding to the path~$\overline{x}$.
    
    More generally, for $f$ an automorphism of~$T$, since the layers are invariant under~$f$, for $u,v\in X^*$, the equation
    \[
    f(uv)=f(u)f|_u(v)
    \]
    defines a unique automorphism $f|_u$ of $T$ called the \emph{section of $f$ at $u$}. This automorphism can be viewed as the automorphism of $T$ induced by $f$ upon identifying the rooted subtrees of $T$ at the vertices $u$ and $f(u)$ with the tree $T$. For $G$ a subgroup of $\Aut T$, we will denote the set of all sections of group elements at $u$ by $G|_u$.
    
The action of $\Aut T$ on~$T$ will be on the left. We observe that, for any $u,v\in X^*$ and automorphisms $f,g\in\Aut T$, we have 
    \begin{align*}
       ( f|_u)|_v&= f|_{uv},\\
       (fg)|_u &=f|_{g(u)} g|_{u},\\
       f^{-1}|_u &= (f|_{f^{-1}(u)})^{-1}.
    \end{align*}
    The corresponding equations also hold for the labels $f|^u$.

\subsection{Subgroups of $\Aut T$}
Let $G$ be a subgroup of $\Aut T$ acting \textit{spherically
  transitively}, that is, transitively on every layer of $T$. The
\textit{vertex stabiliser} $ \mathrm{st}_G(u)$ is the subgroup
consisting of elements in~$G$ that fix the vertex~$u$.  For
$n \in \mathbb{N}$, the \textit{$n$\textsuperscript{th} level stabiliser}
  $\mathrm{St}_G(n)= \bigcap_{\lvert u\rvert =n}
  \mathrm{st}_G(u)$
is the subgroup consisting of automorphisms that fix all vertices at
level~$n$.  

We also write $\mathrm{st}_G(u)|_u$ for the restriction  of the vertex stabiliser
$\mathrm{st}_G(u)$ to the subtree $T_u$ rooted at a
vertex~$u$. Since $G$ acts spherically transitively, the vertex
stabilisers at every level are conjugate under~$G$.  The group $G$ is \textit{fractal} if $\mathrm{st}_G(u)|_u$ coincides with $G|_u$ for all vertices~$u$.

Recall the cyclic subgroup $A_m$ of $\Sym(X)$ generated by $(0 \; 1\; \cdots\; m-1)$. We denote by $\Gamma$ the subgroup of all automorphisms, whose labels are all contained in~$A_m$. In other words, the group $\Gamma$ is the inverse limit  of $n$-fold iterated wreath products of~$A_m$: 
		$$
		\Gamma= \varprojlim_{n\in\mathbb{N}} A_{m} \wr \overset{n}\cdots \wr A_{m}.
$$

A group $G \leq \Aut T$ is called \emph{reducing with respect to $(\mathcal N_{n})_{n \in
\N}$}, if there is a sequence of finite sets $\mathcal N_n \subset \Aut T$ such that for all $g \in G$ there is a positive integer $N$ such that $g|_u \in \mathcal N_m$ for all $u \in X^m$ with $m \geq N$. The sequence $(\mathcal N_{n})_{n \in \N}$ is called the \emph{nuclear sequence of $G$}. If there is some $k \in \N$ such that the sequence $(\mathcal N_{n})_{n \geq k}$ is constant, then $G$ is called \emph{contracting} and the set $\mathcal N = \mathcal N_{k}$ is called the \emph{nucleus} of~$G$.

Let $u \in T$ be a vertex. The \emph{rigid vertex stabiliser of $u$} is the subgroup $\text{rst}(u) \leq \Aut T$ consisting of a automorphisms whose portrait is trivial outside of~$T_u$, and the \emph{$n$\textsuperscript{th} rigid level stabiliser $\Rst(n)$} is the product of all rigid vertex stabilisers of vertices at the $n$\textsuperscript{th} level. A group $G \leq \Aut T$ is called \emph{weakly branch} if it acts spherically transitive and $G \cap \Rst(n)$ is non-trivial for all $n \in \N$. It is called \emph{branch} if it is weakly branch and $G \cap \Rst(n)$ is of finite index in $G$ for all $n \in \N$.

\subsection{Polyspinal groups}\label{sec:multi-spinal}

Spinal groups were first introduced by Bartholdi and \u{S}uni\'{k} in~\cite{BS01} as a common generalisation of the Grigorchuk group and the class of GGS-groups. A more general
definition was formulated by Bartholdi, Grigorchuk and \u{S}uni\'{k} in~\cite{BGS03}.
Below we define a natural generalisation of spinal groups acting on the $m$-adic tree~$T$.

Let $(a_n)_{n \in \N}$ be any sequence. We denote the shift operator $(a_n)_{n \in \N} \mapsto (a_{n+1})_{n \in \N}$ by $\sigma$.

A directed automorphism $f$ is described by a sequence of $(m-1)$-tuples of
permutations of $X$ and a path $\overline{x}$. More generally, given a group $D$, a directed path $\overline{x}$ and a sequence 
$\omega = (\omega_n)_{n \in \N}$ of homomorphisms $\omega_n: D \to \Sym(X)^{X\setminus\{x\}}$, we (recursively) define a tree automorphism for every $d \in D$ by
\[
	d_{\omega, x}|_{y} = \begin{cases}
		d_{\sigma \omega, x} &\text{ if }{y = x},\\
		\pi_{y}\omega_1(d) &\text{ otherwise.}
	\end{cases}
\]
Here $\pi_{y}$ denotes the projection to the ${y}$\textsuperscript{th} component, for $y\in X$.
Write $D_{\omega, x}=\{d_{\omega,x}\mid d\in D\}$ for the set of all such automorphisms. Since $d_{\omega, x}$ fixes
$\overline{x}$, we have $d'_{\omega, x}d''_{\omega, x} = (d'd'')_{\omega, x}$ for all
$d', d'' \in D$, and hence a homomorphism $D \to D_{\omega, x}$ with kernel
\[
	\bigcap_{n \in \N} \ker(\omega_n).
\]
All sections $D_{\omega, x}|_{x^k}$ are isomorphic to $D$ if and only if
$\bigcap_{n \geq k} \ker(\omega_n) = 1$ for all $k \in \N$. In this case we call $D_{\omega, x}$ a \emph{directed group defined by $\overline x$ and $\omega$} and drop the indices, identifying it with $D$.

Let $D$ be a directed group. For every $n \in \N$ we define the $n$\textsuperscript{th} rooted companion group
\[
	\sigma^n R(D) = \langle d|^v \mid |v| = n \rangle = \langle d|^{x^{n-1}y} \mid y \in X \rangle.
\]
To shorten the notation, we write
\begin{align*}
	\sigma^n d &= d|_{x^n}& \text{ for } d \in D, \,n \in \N,\\
	\sigma^n D &= D_{\sigma^n \omega, x}& \text{ for } n \in \N.
\end{align*}

\begin{definition}
	Let $R$ be a group of rooted automorphisms acting transitively on~$X$, and for some $r \in [1, m]$, let $x^{(0)}, \dots, x^{(r-1)}$ be $r$ distinct elements in $X$. Let $D^{(0)}, \dots, D^{(r-1)}$ be directed groups defined by the constant paths given by ${x^{(0)}}, \dots, {x^{(r-1)}}$ and $\omega^{(0)}, \dots, \omega^{(r-1)}$ respectively, where the latter are sequences of homomorphisms $\omega^{(i)}_n: D^{(i)} \to \Sym(X)^{X\setminus\{x^{(i)}\}}$ such that
	\begin{enumerate}
		\item the groups $\sigma^n R(D^{(i)})$ for $i \in [0,r)$ act transitively on $X$ for all $n \in \N$, and
		\item for all $i \in [0,r)$ and $k \in \N$
		$$\bigcap_{n \geq k} \ker(\omega^{(i)}_n) = 1.$$
	\end{enumerate}
	Then
	\[
		G = \langle R, D^{(i)} \mid i \in [0,r) \rangle
	\]
	is called the \textit{polyspinal} group with data $R$, $\omega^{(0)}, \dots, \omega^{(r-1)}$, and $x^{(0)}, \dots, x^{(r-1)}$. If $r = 1$ we drop the superscript $(0)$, and call $G$ the \textit{spinal group with data $R, \omega$}.
\end{definition}

\begin{remark}\label{rem:paths}
	The choice of the path $\overline{x}$ does not matter in case of spinal groups, which is why it is omitted from the defining data. More generally, one easily defines \emph{directed elements along~$\ell$}, an arbitrary infinite simple rooted path in~$T$, and, with this in mind, one can more generally define a spinal group to be equipped with an arbitrary directed path. However, by conjugating by an appropriate element $f \in \Aut T$, we can always assume that $\ell=\overline 0$.
	The same cannot be said about polyspinal groups with more than one arbitrary, non-constant, directed path; compare~\cite[Lemma~2.3]{FAGT}. We will not consider this more general case here.
\end{remark}
	
	For any $n \in \N$, the $n$\textsuperscript{th} shifted companion of $G$, 
	\[
		\sigma^n G = \langle \sigma^n R, \sigma^n D^{(i)} \mid i \in [0,r) \rangle,
	\]
	where $\sigma^n R := \langle \sigma^n R(D^{(i)}) \mid i \in [0,r) \rangle$, is again a polyspinal group.


\begin{definition}
	We record some previously studied special cases:
	\begin{itemize}
		\item If $\sigma^n R$ is equal to the group {$A_m = \langle (0 \; 1 \;\cdots\; m-1) \rangle$} for all~$n$, all $D^{(i)}$ are necessarily direct products of cyclic groups of order~$m$. In this case, we drop the rooted group from the defining data. Assume further that the sequences $\omega^{(i)}$ are all constant. In this case, one calls $G$ a \emph{multi-EGS group}; compare~\cite{KT}.  Clearly, all multi-EGS groups are subgroups of~$\Gamma$. 
		\item A group that is both spinal and a multi-EGS group is called a \emph{multi-GGS group}; compare~\cite{AKT}.
		\item A multi-GGS group such that the unique non-trivial $D^{(0)}$ is cyclic is called a \emph{GGS-group}.
	\end{itemize}
\end{definition}

\begin{lemma}\label{lem:mes reducing}
	Let $G$ be a polyspinal group with defining data $R$, $\omega^{(0)}, \dots, \omega^{(r-1)}$, and $x^{(0)}, \dots, x^{(r-1)}$. Then $G$ is reducing with respect to
	\[
		\left( {\sigma^nR}\cup\bigcup_{i = 0}^{r-1} \sigma^n D^{(i)}  \right)_{n \in \N}.
	\]
\end{lemma}

\begin{proof}
	Every element $g \in G$ can be represented by a word of the form
	\(
		\big(\prod_{j = 0}^{l-1} d_j^{\,r_j}\big) r_{l}
	\)
	for some $l\in\N$, with $r_j \in R$ and $d_j \in D^{(i_j)}$ with $i_j \in [0,r)$, for $j \in [0,l]$. We further assume that $i_j\ne i_{j+1}$ if $r_j=r_{j+1}$, for some $j\in[0, l)$. An element of the form $d_j^{\,r_j}$ is called a \emph{syllable}, and consequently $l = \operatorname{syl}(g)$ is called the \emph{syllable length} of $g$. It is enough to prove $\operatorname{syl}(g|_{xy}) < \operatorname{syl}(g)$ for all $g \in G$ with $\operatorname{syl}(g) > 1$ and $xy \in X^2$.
	
	Let $d_j^{\,r_j}d_{j+1}^{\,r_{j+1}}$ be two neighbouring syllables. Then
	\[
		(d_j^{\,r_j}d_{j+1}^{\,r_{j+1}})|_x = d_j|_{x^{r_j}}d_{j+1}|_{x^{r_{j+1}}}
	\]
	has syllable length one or zero if $i_j^{\,r_j^{-1}} \neq i_{j+1}^{\,r_{j+1}^{-1}}$. Otherwise, for $x=i_j^{\,r_j^{-1}} = i_{j+1}^{\,r_{j+1}^{-1}}$ it is $(d_j^{\,r_j}d_{j+1}^{\,r_{j+1}})|_x = d_jd_{j+1}$. If $i_j = i_{j+1}$, this is again of syllable length $1$. Hence we may assume $i_j \neq i_{j+1}$. But then
	\[
		(d_j^{\,r_j}d_{j+1}^{\,r_{j+1}})|_{xy} = (d_jd_{j+1})|_y = d_j|_y d_{j+1}|_y
	\]
	has syllable length at most $1$.
	
	We have proven that upon taking sections at vertices of level $2$ at most every second syllable may contribute a syllable to the section. Hence $\operatorname{syl}(g|_{xy}) < \operatorname{syl}(g)$ if $\operatorname{syl}(g) > 1$.
\end{proof}

\begin{lemma}
	Polyspinal groups are fractal.
\end{lemma}

\begin{proof}
    Since $\sigma G=G|_u$, where $u$ is any first-level vertex, is again a polyspinal group, it suffices to show that $\mathrm{st}_G(u)|_u$ equals $\sigma G$. The result follows from the definition of $\sigma G$ and upon considering $\langle D^{(0)}, \dots, D^{(r-1)}\rangle^G|_u$. 
\end{proof}

\section{Conditions to be conjugate}\label{sec:conditions}

\subsection{Necessary conditions for spinal groups to be conjugate}

Here, let $G$, respectively $\widetilde G$, denote polyspinal groups with defining data $R$, $\omega^{(0)}, \dots, \omega^{(r-1)}$,  $x^{(0)}, \dots, x^{(r-1)}$, respectively $\widetilde R$, $\widetilde \omega^{(0)}, \dots, \widetilde \omega^{(\widetilde r-1)}$, $\widetilde x^{(0)}, \dots, \widetilde x^{(\widetilde r-1)}$. To be consistent, we write $\sigma^n \widetilde R$ for the rooted generators of
$\sigma^n \widetilde G$.

\begin{lemma}\label{lem:cosets}
	Let $G$ and $\widetilde G$ be polyspinal groups that are conjugate via $f \in \Aut T$; that is, $G^f=\widetilde{G}$. Then $f|_u \equiv f|_v \pmod {\sigma^n G}$ for all $n \in \N$ and $u, v \in X^n$. In particular, $f|^u \equiv f|^v \pmod {\sigma^n R}$.
\end{lemma}

\begin{proof}
	Let $n \in \N$ and $u, v \in X^n$. Since $\widetilde G$ acts spherically transitively, there is an element $g' \in \widetilde G$ such that $u^{g'} = v$. Now $g'|_u \in \sigma^n \widetilde G$, and since $\mathrm{st}_{\widetilde G}(u)|_u = \sigma^n \widetilde G$ there is an element $g'' \in \mathrm{st}_{\widetilde G}(u)$ such that $g''|_u = (g'|_u)^{-1}$. Thus $g = g'g''$ maps $u$ to $v$ and $g|_u = 1$. Let $h \in G$ be such that $h^f = g$. Then, recalling that the action on the tree is on the left, we have $hf(u)=f(f^{-1}hf)(u)=f(v)$. Thus,
	\[
		1 = g|_u = (h^f)|_u = f^{-1}|_{hf(u)} h|_{f(u)} f|_{u} = f^{-1}|_{f(v)} h|_{f(u)} f|_{u} =f|_v^{\,-1} h|_{f(u)} f|_u.
	\]
	Restricting to the label at the vertex~$u$ yields the second statement.
\end{proof}
%
%

\begin{lemma}\label{lem:descent}
	Let $G=\langle R, D\rangle$ be a spinal group directed along $\overline{0}$ and $H \leq \Aut T$ be reducing with respect to $(\mathcal N_{n})_{n \in \N}$. If there is some $f \in \Aut T$ such that $H^f = G$, then there is a positive integer $N$ such that for all $n \geq N$,
	\[
		\sigma^n  D \subseteq \mathcal N_n ^{\,f|_{0^n}}.
	\]
\end{lemma}

\begin{proof}
	For every $d \in D$ there is some element $h(d) \in H$ such that $h(d)^f = d$. Since $H$ is reducing, there is a positive integer $N$ such that for all $d\in D$, we have $h(d)|_u \in \mathcal N_n$ for all $u \in X^n$ with $n\geq N$.
Since $d$ fixes $\overline 0$, the conjugate $h(d)$ must fix $f(\overline 0)$. Thus
	\[
		\sigma^n  d = d|_{0^n } = h(d)^f|_{0^n } = (h(d)|_{f(0^n )})^{f|_{0^n }}\in \mathcal N_n ^{\,f|_{0^n }}.\qedhere
	\]
\end{proof}

\begin{lemma}\label{lem:dir}
	Let $D$ and $\widetilde D$ be two directed groups defined by $\overline x, \omega$ and $\overline y, \widetilde \omega$ respectively. If there exists an automorphism $f \in \Aut T$ such that $\widetilde D^f = D$, then there exists an isomorphism $\iota: D \to \widetilde D$ such that for all $n \in \N$,
	\[
		\omega_n = (c(f|^{x^{n-1} z}))_{z \in X \setminus\{x\}} \circ p(f|^{x^{n-1}}) \circ\widetilde \omega_{n} \circ  \iota,
	\]
	where $c(\alpha)$ is the inner automorphism induced by $\alpha$ and $p(\alpha)$ is the relabelling of a direct product $\Sym(X)^{X\setminus\{y\}}$ by $\alpha$ for any $\alpha \in \Sym(X)$, i.e.\ the $k$\textsuperscript{th} component of $(\Sym(X)^{X\setminus\{y\}})^{p(\alpha)}=\Sym(X)^{X\setminus\{\alpha^{-1}(y)\}}$ is the $\alpha(k)$\textsuperscript{th} component of $\Sym(X)^{X\setminus\{y\}}$.
\end{lemma}

\begin{proof}
	Denote by $\iota$ the isomorphism induced by $D = \widetilde D^f$. Since all elements of $D$ have labels only at distance $1$ from $\overline x$, respectively all elements of $\widetilde D$ have labels only at distance $1$ from $\overline y$, we have $f(\overline x) = \overline y$. For any $d \in D$, $z \in X$ and $n \in \N$, it follows that
		\begin{align*}
			\begin{rcases}
				\pi_{z}\omega_{n}(d) &\text{for } z \not = x\\
				\sigma^n d &\text{for } z = x
			\end{rcases}
			&= \sigma^{n-1} d|_{z} \\
			&= d|_{x^{n-1}z}\\
			&= \iota(d)^{f}|_{x^{n-1}z}\\
				&= (\sigma^{n-1} \iota(d))^{f|_{x^{n-1}}}|_{z}\\
			&=
			\begin{cases}
				\big(\pi_{f|^{i^{n-1}}(z)}\widetilde\omega_{n}(\iota(d))\big)^{f|_{x^{n-1} z}} &\text{for } f|^{x^{n-1}}(z) \not = y \Leftrightarrow z \neq x,\\
				\big(\sigma^n \iota(d)\big)^{f|_{x^{n-1} z}} &\text{for } f|^{x^{n-1}}(z) = y \Leftrightarrow z = x.
			\end{cases}
		\end{align*}
		Hence the result.
\end{proof}

\smallskip

\begin{proof}[Proof of \cref{Theorem}{thm:spinal}]
	By \cref{Lemma}{lem:mes reducing} and \cref{Lemma}{lem:descent} there is $N \in \N$ such that
	\[
		\sigma^n  D \subseteq (\sigma^n  \widetilde D)^{f|_{0^n }} \cup {(\sigma^n  \widetilde R)}^{f|_{0^n }},
	\]
	for all $n \geq N$. But since $\sigma^n  D \leq \St(1)$ for all such $n$, one obtains
	\(
		\sigma^n  D \leq (\sigma^n  \widetilde D)^{f|_{0^n }}.
	\)
	By symmetry (possibly increasing~$N$) we have equality for all $n \geq N$. By~\cref{Lemma}{lem:dir} we have
	\[
		\omega_n  = (c(f|^{0^{n-1}  x}))_{x \in X \setminus\{0\}} \circ p(f|^{0^{n-1}}) \circ \widetilde \omega_n  \circ  \iota
	\]
	for some isomorphism $\iota: D \to \widetilde D$ and all $n  \geq N$.
	By \cref{Lemma}{lem:cosets} we may write $f|^{0^{n-1}  x} = r_x f|^{0^n }$ for some $r_x \in \sigma^n  \widetilde R$. For all $x\in X\setminus \{0\}$, it follows that
	\begin{equation}\label{eq:1}
		\pi_x\omega_n (d) = (\pi_{f|^{0^{n-1}}(x)}\widetilde\omega_n  (\iota(d)))^{f|^{0^{n-1}  x}} = (\pi_{f|^{0^{n-1}}(x)}\widetilde\omega_n  (\iota(d)))^{r_x f|^{0^n}},
	\end{equation}
	implying $(\sigma^n \widetilde R)^{f|^{0^n}} = \sigma^n R$. Thus for every $n  \ge N$ we choose
	\begin{itemize}
		\item the inner automorphism of $\St_{\Sym(X)}(0)$ induced by $f|^{0^n }$ as $\phi_n $,
		\item the inner automorphisms of $\sigma^{n}\widetilde R$ induced by the $r_x$ as $\rho_n $, and
		\item the automorphism $p(f|^{0^{n-1} })$ induced by the permutation $f|^{0^{n-1} }$ as $\alpha_n $.
	\end{itemize}
	Then \eqref{eq:1} implies the statement.
\end{proof}

\subsection{Multi-EGS and multi-GGS groups}

Recall the group $\Gamma$, which is the inverse limit of $n$-fold iterated wreath products of $A_m = \langle (0\;1\;\cdots\; m-1) \rangle$.

\begin{lemma}\label{lem:conjugating spinal elements}
	Let $g, h \in \Gamma$ be directed elements along $\overline{x}$, and $f \in \Aut T$ such that $g^f \in \Gamma$ is directed along $\overline{y}$, for some $x, y \in X$. Then $h^f$ is directed along $\overline{y}$.
\end{lemma}

\begin{proof}
	Without loss of generality one can consider the case $x = y = 0$. Since $g^f|_0$ is directed, the element $f$ stabilises~$0$. For any $x \in X\setminus \{0\}$ there is $n \in \Z/m\Z$ such that
	\[
		g^f|_x = (g|_{f(x)})^{f|_x} = ((0\; 1 \cdots m-1)^n)^{f|_x}.
	\]
	Since $g^f$ is directed and a member of $\Gamma$,
	\[
		g^f|_x \in A_m,
	\]
	hence $f|_x$ normalises $A_m$. Thus
	\[
		h^f|_x = \begin{cases}
			(h|_0)^{f|_0} &\text{ if }x = 0,\\
			(h|_{f(x)})^{f|_x} \in A_m &\text{ otherwise.}
		\end{cases}
	\]
	Repeating this argument for $g|_0, h|_0$ and $f|_0$ shows that $h^f$ fixes $\overline 0$ and has non-trivial labels only at vertices of distance $1$ to this ray. Thus $h^f$ is directed along~$\overline 0$, as required.
\end{proof}

Recall that for a multi-EGS group $G$, for $i\in X$ the directed groups $D^{(i)}$ are direct products of cyclic groups of order~$m$ and the rooted group is equal to~$A_m$.

\begin{proposition}\label{prop:mEGSness}
	Let $G$ and $\widetilde G$ be multi-EGS groups defined by $\omega$ and $\widetilde \omega$, respectively, such that $\widetilde G^f = G$ for an element $f \in \Aut T$. Then 
	for $i \in [0,r)$, there exist
   	\begin{itemize}
   		\item automorphisms $\alpha^{(i)}$ of $A_m^{\,X\setminus\{\widetilde x^{(i)}\}}$ permuting the direct factors by an element of $\Norm_{\Sym(X)}(A_m)\cap\St_{\Sym(X)}(\widetilde x^{(i)})$,
		\item a map $\theta: [0,r) \to [0, \widetilde r)$ such that $\rank D^{(i)} = \rank \widetilde D^{(\theta(i))}$, and
   		\item isomorphisms $\iota^{(\theta(i))}: D^{(i)}\rightarrow \widetilde D^{(\theta(i))}$ such that
   	\end{itemize}
	\[
		\omega^{(i)} = \alpha^{(\theta(i))} \circ \widetilde \omega^{(\theta(i))} \circ  \iota^{(\theta(i))}.
	\]
\end{proposition}

\begin{proof}
	Let $i \in [0,r)$. Then the group generated by $A_m$ and $D^{(i)}$ is a multi-GGS group. Write $x$ for $x^{(i)}$, for $x\in X$. By \cref{Lemma}{lem:descent} and \cref{Lemma}{lem:mes reducing}, recalling that $\sigma D^{(i)}=D^{(i)}$ and $\sigma R=R=A_m$ for a multi-EGS group, there is some $n \in \N$ such that 
	\[
		 D^{(i)} \subseteq \left( A_m\cup \bigcup_{j = 0}^{\widetilde r}  \widetilde D^{(j)} \right)^{f|_{x^n}}.
	\]
	Since $D^{(i)}$ stabilises the first layer, in fact
	\[
		D^{(i)} \subseteq \left(\bigcup_{j = 0}^{\widetilde r}  \widetilde D^{(j)} \right)^{f|_{x^n}}.
	\]
	Let $d \in D^{(i)}$ and $e \in \widetilde D^{(j)}$, for some $j \in [0, \widetilde r)$,  be such that $e^{f|_{x^n}} = d$. Then
	\[
		d = d|_x = (e^{f|_{x^n}})|_x = (e|_{f|^{x^n}(x)})^{f|_{x^{n+1}}}.
	\]
Now write $y$ for $\widetilde x^{(j)}$.
Since there is only one non-trivial first layer section of $e$ stabilising the first layer, this implies $y= f|^{x^n}(x)$. Defining $\theta(i) = j$ we have
	\[
		D^{(i)} \leq (\widetilde D^{(\theta(i))})^{f|_{x^n}}.
	\]
	Now let $e \in \widetilde D^{(\theta(i))}$. By  \cite[Lemma 3.3]{Moritz} and the fact that multi-EGS groups are fractal, we obtain $e^{f|_{x^n}} \in G$. However by \cref{Lemma}{lem:conjugating spinal elements} it follows that since there are elements 
	$e' \in \widetilde D^{(\theta(i))}$ such that $e'^{f|_{x^n}} \in D^{(i)}$ is $\overline x$-spinal, the element $e^{f|_{x^n}}$ is $\overline x$-spinal. Hence by \cref{Lemma}{lem:mes reducing} there is a positive integer $k(e)$ such that $(e^{f|_{x^n}})|_{x^{k(e)}} \in D^{(i)}$. Thus
	\[
		(e^{f|_{x^n}})|_{x^{k(e)}} = (e|_{f|^{x^n}(x^{k(e)})})^{f|_{x^{n + {k(e)}}}} \in \St_G(1)\setminus\{1\},
	\]
	hence $f|^{x^n}(x^{k(e)}) = y^{k(e)}$ and $(e^{f|_{x^n}})|_{x^{k(e)}} = e^{f|_{x^{n + {k(e)}}}} \in D^{(i)}$. 
	
	Set $k_{\mathrm{max}} = \max \{k(e) \mid e \in \widetilde D^{(\theta(i))} \}$ to obtain
	\[
		(\widetilde D^{(\theta(i))})^{f|_{x^{n + k_{\mathrm{max}}}}} \leq D^{(i)}.
	\]
	Hence $D^{(i)} = (\widetilde D^{(\theta(i))})^{f|_{x^{n  + k_{\mathrm{max}}}}}$. Taking further sections it is clear that 
	\[
		D^{(i)} = (\widetilde D^{(\theta(i))})^{f|_{x^k}}
	\]
	for all $k \geq n + k_{\mathrm{max}}$.

	Since all directed groups involved are abelian and both rooted groups are equal and cyclic, we have by \cref{Lemma}{lem:dir} that for all $i \in [0, r)$,
	\[
		\omega^{(i)} = (\phi^{(\theta(i))})^{X \setminus \{x\}} \circ \alpha^{(\theta(i))} \circ \widetilde \omega^{(\theta(i))} \circ  \iota^{(\theta(i))}
	\]
	with $\alpha^{(\theta(i))}$ as desired, and $\phi^{(\theta(i))} \in \Norm_{\Sym(X)}(A_m) \cap \St(x) \cong (\Z/m\Z)^{\times}$. Hence, writing $c=(0\;1 \;\cdots \;m-1)$, we have $\phi^{(\theta(i))}(c) = c^{k^{(\theta(i))}}$ for some $k^{(\theta(i))}$ coprime to $m$. 
	Therefore, replacing $\iota^{(\theta(i))}$ with
	\[
	\hat\iota^{(\theta(i))}=\mu_{k^{(\theta(i))}}\circ \iota^{(\theta(i))}:D^{(i)}\rightarrow \widetilde D^{(\theta(i))}
	\]
	where $\mu_{k^{(\theta(i))}}(\iota^{(\theta(i))}(d))=(\iota^{(\theta(i))}(d))^{k^{(\theta(i))}}$, we obtain
	\[
		\omega^{(i)} =  \alpha^{(\theta(i))} \circ\widetilde \omega^{(\theta(i))} \circ  \hat\iota^{(\theta(i))},
	\]
	as required.
\end{proof}

\begin{lemma}\label{lem:sufficient}
	Let $G
	$ be a multi-GGS group defined by $\omega$, and let $\iota \in\Aut(D)$, and $\alpha \in \Norm_{\Sym(X)}(A_m) \cap \St_{\Sym(X)}(0)$. Then the multi-GGS group $\widetilde G$ defined by $p(\alpha) \circ\omega \circ  \iota$ is conjugate to $G$.
\end{lemma}

\begin{proof}
	The multi-GGS group defined by $\omega \circ \iota$ is equal to $G$, as for any $d \in D$,
	\[
		\iota(d)_\omega = d_{\omega \circ \iota }
	\]
	and vice versa.
	
	Let $\kappa = \kappa(\alpha)$ be the automorphism with constant portrait $\alpha$. Then
	\[
		d^\kappa = (d^\kappa, (d|_{1^\alpha})^{\kappa}, \dots, (d|_{(m-1)^\alpha})^{\kappa})
	\]
	for all $d \in D$, hence $d^\kappa = d_{ p(\alpha) \circ\omega \circ \hat\iota}$, using the notation $\hat\iota$ from the previous lemma. Since $\kappa$ centralises rooted elements, we have $G^\kappa = \widetilde G$.
\end{proof}

\begin{proof}[Proof of \cref{Theorem}{thm:mGGS}]
	The necessarity of the condition is a direct consequence of \cref{Proposition}{prop:mEGSness}. The sufficiency follows from \cref{Lemma}{lem:sufficient}.
\end{proof}

\section{Finitely generated non-spinal branch groups}\label{sec:non-spinal}

We now prove that one of the (finitely generated branch) \emph{Extended Gupta--Sidki groups} defined by Pervova~\cite{Per07} is not isomorphic to a weakly branch spinal group. The attentive reader will notice that this by far is not the only example for a group with this property within the class of polyspinal groups.

Let $a = (0 \; 1 \; 2)$ be rooted and define $b = (b, a, a^2), c = (a^2, c, a)$. The group 
$G = \langle a, b, c \rangle$ is a polyspinal group with defining data $R = \langle a
\rangle \cong A_3$ and $\omega^{(0)}_n = (1: x \mapsto x, \,2: x \mapsto x^{-1})$, $\omega^{(1)}_n = (0: x \mapsto x^{-1}, \,2: x \mapsto x)$ for all $n \in \N$, where the letter in front of a colon signifies the component of the homomorphism after it. Both $\langle a, b \rangle$ and $\langle a, c \rangle$ are isomorphic to the Gupta--Sidki $3$-group. We record some of the results of~\cite{Per07} in a lemma.

\begin{lemma}\label{lem:egs}
The group~$G$ is a just infinite torsion branch group.
\end{lemma}

We now show the following.

\begin{lemma}\label{lem:EGS not conjugate to spinal}
	The group~$G$ is not conjugate to any spinal group $S \leq \Aut T_3$.
\end{lemma}

\begin{proof}
	Assume for contradiction that $G^f = \widetilde G=\langle R, D\rangle$ is a spinal group, for some $f\in\Aut T_3$. Clearly, the rooted group of~$\widetilde G$ must be cyclic of order~$3$. Following our usual strategy, by \cref{Lemma}{lem:descent} and \cref{Lemma}{lem:mes reducing} we find $n \in \N$ such that 
	\[
		 D_{\sigma^n \omega} \subseteq (\langle a \rangle \cup \langle b \rangle \cup \langle c \rangle)^{f|_{0^n }}.
	\]
	Since $D_{\sigma^n \omega}$ stabilises the first layer, we have $D_{\sigma^n \omega} \cap \langle a \rangle^{f|_{0^n }} = 1$. Let $d_{\sigma^n \omega} \in D_{\sigma^n \omega}$ equal  $(c^i)^{f|_{0^n }}$ for some $i \in \F_3$. Then, recalling that $c\in \St(1)$, we obtain 
	\[
		d_{\sigma^{n +1}\omega} = d_{\sigma^n \omega}|_0 = (c^i)^{f|_{0^n }}|_0 =(f|_{0^n })^{-1}|_{f|_{0^n}(0)}\,c^i\mid_{f|_{0^n}(0)}f|_{0^{n+1} }=(c^i\mid_{f|_{0^n}(0)})^{f|_{0^{n+1} }}.
	\]
As $d_{\sigma^{n +1}\omega}\in \St(1)$, it follows that $f|_{0^n}(0)=1$. Repeating the argument with some $e_{\sigma^n\omega}\in D_{\sigma^n\omega}$, which equals  $(b^j)^{f|_{0^n }}$ for some $j \in \F_3$, we see that
		\[
		e_{\sigma^{n +1}\omega} = e_{\sigma^n \omega}|_0 = (b^j)^{f|_{0^n }}|_0 =(b^j\mid_{1})^{f|_{0^{n+1} }}= (a^j)^{f|_{0^{n +1}}} \in \St(1),
	\]
	hence $j = 0$ and $e_{\sigma^n \omega} = 1$. It follows $D_{\sigma^n \omega} \leq \langle c \rangle^{f|_{0^n }}$. Thus $D \cong C_3$ and $\widetilde G$ is generated by two elements. But since $G/\St_G(1)' \cong C_3 \wr C_3^{\,2}$ the group $G$ cannot be two-generated. This is a contradiction.
\end{proof}

\begin{proof}[Proof of \cref{Theorem}{thm:non-spinal}]
	By \cref{Lemma}{lem:egs}, the group~$G$ is branch, torsion, and just infinite. If $S \leq \Aut T_3$ is isomorphic to $G$, it is conjugate to $G$ by \cite[Corollary 1(a)]{GrigorchukWilson} and \cite[Proof of Corollary 3.8]{KT}. 
	However by \cref{Lemma}{lem:EGS not conjugate to spinal}, this is impossible.
	
	If $S \leq \Aut \widetilde T$ is weakly branch (with respect to its embedding into $\Aut \widetilde T$) and isomorphic to $G$, it is just infinite and hence branch. Thus by \cite[Theorem 2]{GrigorchukWilson} we have $\widetilde T \cong T_3$ and the first assertion implies the second.
\end{proof}

We remark that all known spinal groups that are not weakly branch act on the binary tree such that the sequence of companion groups stabilises as the infinite dihedral group; compare \cite[Proposition~3.4]{NT}. Therefore we ask: ``Is every involution-free spinal group weakly branch with respect to its natural embedding?''

It is natural to update \cite[Question 4]{BGS03} to ``Does every finitely generated branch group admit an embedding into some $\Aut T$ as a branch polyspinal group?''. To conclude this paper, we will make some remarks concerning the candidate presented in \cite{BGS03}, which is the perfect regular branch group defined by Peter Neumann \cite{Neu86}. For a short description see \cite[Section 1.6.6]{BGS03}.

\begin{proposition}
	Neumann's group $G \leq \Aut T$ is not conjugate to any spinal group $S \leq \Aut T$.
\end{proposition}

\begin{proof}
	Assume for contradiction that $G^f$ is a spinal group with data $R, \omega$. Then
	by \cref{Lemma}{lem:descent} there is a number $n \in \N$ such that
	\[
		\sigma^n D \subseteq \mathcal N^{f|_{0^n}}.
	\]
	As the elements of $\sigma^n D$ stabilise the first layer, we arrive at a contradiction since $\mathcal{N}$ consists of elements that do not belong to $\St(1)$.
\end{proof}

By a rigidity result of Lavreniuk and Nekrashevych \cite[Section 8]{LN02} the automorphism group of Neumann's group $G$ coincides with its normaliser in $\Aut T$, but further results allowing us to reduce any isomorphism to a subgroup of $\Aut T$ (or even $\Aut \widetilde T$) would be necessary to negatively answer the updated question.

\end{document}